\documentclass[10pt]{amsart}

\usepackage{amsthm}
\usepackage{amscd}
\usepackage[final]{graphics}
\usepackage{amssymb,amsmath}  %
\usepackage{latexsym} %
\usepackage[english]{babel}
\usepackage{rotating}%

\usepackage{tabularx}
\usepackage{longtable}

\newtheorem{lem}{Lemma}%
    \newtheorem{prop}[lem]{Proposition}
    \newtheorem*{thm}{Theorem}
    
    \newtheorem*{cor}{Corollary}

   \theoremstyle{definition}
    \newtheorem{dfn}[lem]{Definition}%
\newtheorem{ntn}[lem]{Notation} 

\theoremstyle{remark}
    \newtheorem{rem}[lem]{Remark}

\def\bb{\mathbf}

\def\phi{\varphi}

\def\calX{\mathcal{X}}

\def\balX{\bar\calX}

\newcommand{\cusub}[2]{{#1}_{\{{#2}\}}}
\newcommand{\cc}[2]{\binom{{#2}+{#1}}{#2}}
\newcommand{\ccc}[2]{c_{{#1},{#2}}}
\newcommand{\Pp}[1]{\bb{P}^{#1}}

\newcommand{\GL}{\mathrm{GL}}

\def\Z{\bb{Z}}
\def\Q{\bb{Q}}

\def\C{\bb{C}}

\def\pu{\bullet}
\newcommand{\s}{\mathfrak S}
\newcommand{\coh}[3][\Q]{H^{#2}({#3};{#1})}
\newcommand{\BM}[3][\Q]{\bar{H}_{#2}({#3};{#1})}

\def\ba{\big|}
\newcommand{\bax}{\ba\calX\ba_\rho}
\DeclareMathOperator{\Sing}{Sing}
\DeclareMathOperator{\reg}{reg}

\newcommand{\B}[2]{B({#2},{#1})}
\newcommand{\F}[2]{F({#2},{#1})}

\author{Orsola Tommasi}
\address{Institut f\"ur Algebraische Geometrie\\
Leibniz Universit\"at Hannover\\
Wel\-fen\-gar\-ten~1\\
D--30167 Hannover\\
Germany}
\email{tommasi@math.uni-hannover.de}
\thanks{Supported by  DFG grant Hu-337/6-2}

\title{Stable cohomology of spaces of non-singular hypersurfaces}

\begin{document}

\begin{abstract}
We prove that the rational cohomology of the space $X_{d,N}$ of non-singular complex  homogeneous polynomials of degree $d$ in $N$ variables stabilizes to the cohomology of $\GL_{N}(\C)$ for $d$ sufficiently large. 
\end{abstract}
\maketitle

\section{Introduction}

Let us fix variables $x_1,\dots,x_N$ and denote by $X_d$ the space of non-singular homogeneous polynomials of degree $d$ in $x_1,\dots,x_N$ with complex coefficients.
A recent result of Vakil and Wood \cite[Thm. 1.13]{vakilwood} about stabilization in the Grothendieck ring suggests that the rational cohomology of $X_d$ stabilizes for $d\gg 0$, in the sense that its $k$th cohomology group  is independent of $d$ for $d$ sufficiently large with respect to $k$.
In this note, we prove that this is indeed the case for $k<\frac{d+1}2$ and describe explicitly the stable cohomology of $X_d$, by proving it is isomorphic to the cohomology of the general linear group.

Let us remark that Peters and Steenbrink proved in \cite{ps-leray} that the rational cohomology of $X_d$ contains a copy of the cohomology of $\GL_N(\C)$ for $d\geq3$. 
Thus the same property should hold for stable cohomology, provided it exists.
More precisely, Peters and Steenbrink showed in \cite{ps-leray} that the cohomology of $X_d$ is isomorphic to the tensor product of the cohomology of $\GL_N(\C)$ and  that of the moduli space $M_d:=X_d/\GL_N(\C)$ of smooth degree $d$ hypersurfaces in $\Pp{N-1}$.
Hence, in view of the results of \cite{ps-leray}, one can say that the stable cohomology of $X_d$ is the minimal possible and that it coincides with the subalgebra generated by the classes described in \cite[\S5--6]{ps-leray}. 
Furthermore, our result implies that the cohomology of the moduli space $M_d$ vanishes in (low) degree $k>0$ if $d$ is sufficiently large.

Our approach to stable cohomology is based on Vassiliev's method \cite{Vbook} for computing the cohomology of complements of discriminants, i.e. of the locus of non-singular elements inside a vector space of functions.
A main feature of this method is the possibility of computing the cohomology of the complement of a discriminant from the description of the possible singular loci of the elements of the discriminant. In this way one obtains a relationship between the topology of the complement of the discriminant and the geometry of the spaces of singular configurations.
We will show that, in the case of $X_d$, this kind of approach yields that stable cohomology is  determined by the geometry of spaces of configurations of up to $N$ points in $\Pp{N-1}$.

Vassiliev's method was adapted to the algebro-geometric setting 
first by Vassiliev in \cite{Vart} and subsequently by Gorinov \cite{gorinov} and by the author \cite{OTM4}.
All these approaches are equivalent for the purpose of this note, which only requires to deal with the first steps of the method. However, for the sake of completeness, we include a direct construction of the part of Vassiliev's spectral sequence we need to prove our result. This construction is strongly based on the construction of cubical hyperresolutions of singular spaces in \cite{hyperres}, see also \cite[\S5]{ps-book}.

It is interesting to remark that in almost all cases in which the rational cohomology of $X_d$ is known, it coincides with the stable cohomology. 
For instance, it is well known that the cohomology of the space $X_d$ for $N=2$ and $d\geq 4$ coincides with the cohomology of $\GL(2)$. A short proof based on Vassiliev's method can be found in \cite[Lemma 5.1]{OTM4}. 
For $N=3$, i.e. for polynomials defining non-singular plane curves, the rational cohomology was computed for degree $3,4$ and $5$ in \cite[Thm. 1 and 2]{Vart} and \cite{gorinov}, respectively. Moreover, the cohomology for $N=4$ and degree $3$, i.e. for cubic surfaces, is described in \cite[Thm. 4]{Vart}. Among these examples, the only case in which non-stable cohomology occurs is for plane quartics.

Finally, let us remark that the results of \cite{vakilwood} in the Grothendieck ring hold for a much larger class of spaces than the spaces $X_d$. Specifically, one would expect cohomological stability phenomena in much greater generality,
for the space of divisors with a prescribed number of singular points in the linear system $|L^{\otimes d}|$ for $L$ a very ample line bundle on an arbitrary smooth manifold $X$.
We plan to consider this more general situation in a subsequent paper.

\subsection*{Notation}
Throughout this note we will make an extensive use of Borel--Moore homology, i.e. homology with closed support, which we denote by the symbol $\bar H_\pu$. For its definition and the properties we refer to \cite[Chapter 19]{intersectiontheory}.

In our results, we take  into account mixed Hodge structures on homology and cohomology groups. The Hodge structures that arise in the stable cohomology of $X_{d}$ are always pure and of Tate type. We will use the notation
$\Q(-k)$ for the Tate Hodge-structure of weight $2k$ (i.e. of Hodge type $(k,k)$).

\subsection*{Acknowledgements}
I would like to thank Ravi Vakil and Melanie Matchett Wood for helpful discussions on this project. Furthermore, I am indebted to the referee for many useful suggestions on how to improve the paper. Finally, I would like to thank Remke Kloosterman for help with the proof of Lemma~\ref{vectorbundle}.

\section{The result}
For a fixed $n\geq 1$, let us denote by $V_{d,n}=\C[x_0,\dots,x_n]_d$ the vector space of homogeneous polynomials of degree $d$. 
Let us recall that a polynomial $f\in V_{d,n}$ is \emph{singular} if there is a non-zero vector $y=(y_0,\dots,y_n)\in\C^n\setminus(0,\dots,0)$ such that all partial derivatives of $f$ vanish at $y$, i.e. if we have 
$$\frac {\partial f}{\partial x_0}(y)=\dots=\frac{\partial f}{\partial x_n}(y)=0.$$
This condition can be interpreted geometrically as follows. If $f\neq0$ holds, then the vanishing of $f$ defines a hypersurface in projective space $\Pp n$. The condition above means that the point $[y]\in\Pp n$ belongs to the singular locus of the hypersurface, i.e. to the closed subset of points at which the hypersurface is not a smooth complex manifold.

The  locus of singular polynomials inside 
$V_{d,n}$ is called the \emph{discriminant hypersurface} $\Sigma=\Sigma_{d,n}$. We denote its complement by $X_{d,n}=V_{d,n}\setminus\Sigma_{d,n}$. 

The aim of this note is to prove the following result:
\begin{thm}
The cohomology with rational coefficients of $X_{d,n}$ and the rational cohomology of $\GL_{n+1}(\C)$ (considered as a topological space) are isomorphic in degree $k< \frac {d+1}2$.
\end{thm}

By work of Peters and Steenbrink, this isomorphism between the cohomology of $X_{d,n}$ and $\GL_{n+1}(\C)$ in low degree is induced by the orbit map associated with the action of the general linear group on $X_{d,n}$. Specifically, their result is the following.

\begin{prop}[{\cite[Lemma~7]{ps-leray}}]\label{prop}
If $d\geq 3$ and $n\geq 1$, then
for every polynomial $f(x_0,\dots,x_n)\in X_{d,n}$, the orbit map 
$$
\begin{array}{rccc}
r:&\GL_{n+1}(\C)&\longrightarrow &X_{d,n}\\
 &g&\longmapsto&f(g(x_0),\dots,g(x_n))
\end{array}
$$
given by the natural action of $\GL_{n+1}(\C)$  on $\C^{n+1}=\operatorname{span}(x_0,\dots,x_n)$
induces a surjection $\coh\pu{X_{d,n}}\rightarrow\coh\pu{\GL_{n+1}(\C)}$ in cohomology.
\end{prop}

In particular, to prove the above theorem it is enough to show that the orbit map is an isomorphism in the stable range $k <\frac {d+1}2$.
Furthermore, Peters and Steenbrink use Proposition~\ref{prop} to prove a stronger result, namely, the degeneration at $E_2$ of the  Leray spectral sequence in rational cohomology associated with the quotient map $X_{d,n}\rightarrow M_{d,n}:=X_{d,n}/\GL_{n+1}(\C)$. In particular, there is an isomorphism
\begin{equation}\label{ps-iso}
\coh\pu{X_{d,n}}\cong \coh\pu{M_{d,n}}\otimes \coh\pu{\GL_{n+1}(\C)}
\end{equation}
of graded $\Q$-vector spaces with mixed Hodge structures.
Note that the quotient space $M_{d,n}=X_{d,n}/\GL_{n+1}(\C)$ is the moduli space of non-singular degree $d$ hypersurfaces in $\Pp n$.

If we combine the theorem above with Peters--Steenbrink's isomorphism \eqref{ps-iso}, we get
\begin{cor}
The rational cohomology of the moduli space $M_{d,n}$ of non-singular degree $d$ hypersurfaces in $\Pp n$ vanishes in degree $k$ for $0<k< \frac {d+1}2$ and $d\geq 3$.
\end{cor}

Before we proceed to outline the structure of the proof of the main theorem, let us discuss some natural questions related to it. 
First, the theorem implies the existence of isomorphisms $\coh k{X_{d,n}}\cong\coh k{X_{d',n}}$ for all $d'>d>2k-1$. However, the proof of the theorem we will give in the next sections does not explain how to construct such stability isomorphisms in a natural way. The reason behind that is that we will only investigate the range in which the $E^1$ terms in Vassiliev's spectral sequence vanish. Although we will not discuss this further, it is indeed possible to construct isomorphisms between non-zero terms of the Vassiliev's spectral sequence and this should expectedly give more information about the existence of stability or transfer maps between the $\coh k{X_{d,n}}$ for different values of $d$ and fixed $k$ in the stable range.

Secondly, although all our results are for cohomology with rational coefficients, we would like to remark that Vassiliev's method can also be applied to compute cohomology with integral coefficients. The vanishing result (Lemma~\ref{main}) on which the proof of the main theorem is based is not expected to hold if one replaces $\Q$-coefficients with $\Z$-coefficients, but it is conceivable that 
one can still detect stability phenomena using Alexander duality \eqref{alex} and the Vassiliev's spectral sequence with integral coefficients. 

Finally, the bound $k<\frac {d+1}2$ for the stability range is not optimal. We will discuss how to improve it in Remark~\ref{betterbounds}.

\medskip

By Proposition~\ref{prop}, to prove the theorem it suffices to prove that the cohomology of $X_{d,n}$ in the degree range $k< \frac {d+1}2$ is not larger than one copy of $\coh\pu{\GL_{n+1}(\C)}$.
The cohomology of $\GL_{n+1}(\C)$ is well known. By \cite[Prop.~7.3]{borel}, it is an exterior algebra generated in odd degree. As shown in \cite[\S5]{ps-leray}, there are exactly $n+1$ generators $\eta_k$ ($k=0,\dots,n$) of degree $2k+1$ and Hodge type $(k+1,k+1)$, with a very precise geometrical description.  
In particular, the cohomology of $\GL_{n+1}(\C)$ vanishes in degree larger than $(n+1)^2$ and is generated by the product $\eta_0\eta_1\dots\eta_{n}$ %
 in degree $(n+1)^2$. To prove the theorem, it suffices to prove that
the cohomology of $X_{d,n}$ in the stable range $k< \frac{d+1}2$ vanishes in degree larger than $(n+1)^2$.

\medskip
By definition, requiring a polynomial $f$ to be singular at a given point $p\in\Pp n$ imposes $n+1$ conditions. Therefore, if we choose $N$ points in $\Pp n$ and require $f$ to be singular at all of them, the naive expectation is that this will impose $N(n+1)$ conditions on $f$. Indeed, this is always the case if $N$ is sufficiently small with respect to the degree of $f$.  

\begin{lem}\label{vectorbundle}
For a fixed integer $N\geq 1$, 
the restriction of 
\begin{equation}\label{vbundle}
\left\{(f,p_1,\dots p_N)\in V_{d,n}\times (\Pp n)^N|\;
p_1,\dots,p_N\in\Sing(f)\right\}\xrightarrow{\pi} (\Pp n)^N
\end{equation}
to the locus where all $p_i$ are distinct is a vector bundle of rank $\cc{d}{n}-N(n+1)$ if and only if $d\geq 2N-1$ holds.
\end{lem}

\proof
To prove the Lemma, we translate its statement into the language of commutative algebra. 
Let us fix $N$ distinct points $p_1,\dots,p_N$ in $\Pp n$ and denote by $I\subset \C[x_0,\dots,x_n]$ the ideal of the set $\{p_1,\dots,p_N\}$. Then the fibre of $\pi$ is exactly the degree $d$ part of the second power $I^2$ of the ideal $I$; its codimension in $V_{d,n}$ is (by definition) the value $P_{M}(d)$ of the Hilbert function of the module $M=\C[x_0,\dots,x_n]/I^2$. 
As the Hilbert polynomial of $M$ is exactly $n(N+1)$, what we want to prove is that the Hilbert function and the Hilbert polynomial of $M$ agree in degree $d$ for $d\geq 2N-1$. 

Let us observe that the depth of $M$ is at least $1$, as both $I$ and $I^2$ are saturated with respect to the irrelevant ideal of $\C[x_0,\dots,x_n]$. 
Hence, the projective dimension of $M$ is $n$ by the Auslander--Buchsbaum formula. By \cite[Thm. 4.2(2)]{syzygies}, this implies that the Hilbert function and the Hilbert polynomial of $M$ are equal if $d$ is larger than or equal to the Castelnuovo--Mumford regularity of $M$. 
Therefore, it suffices to show that the regularity of the module $M$ is at most $2N-1$, or, equivalently, that the regularity of the ideal $I^2$ is at most $2N$. But by \cite[Thm. 1.1]{GGP} (see also \cite{chandler}) we have $\reg(I^2)\leq 2\reg(I)$ and the regularity of the ideal $I$ of $N$ points of $\Pp n$ is at most $N$ (see for instance \cite[Thm. 4.1]{syzygies} and the discussion preceding it). Note that the regularity of $I$ is $N$ if and only if the points $p_1,\dots,p_N$ are collinear.

It remains to prove that the bound $d\geq 2N-1$ is sharp. This follows from an explicit calculation in the case where the points $p_1,\dots,p_N$ lie on the same line in $\Pp N$. In this case, the space of polynomials in $V_{d,n}$ singular at $p_1,\dots,p_N$ has codimension $N(n+1)$ for $d\geq 2N-1$ and codimension at most $N(n-1)+d+1\leq N(n+1)-1$ for $d\leq 2N-2$. 
\qed

To complete the proof of the Theorem, we will show the following:
\begin{lem}\label{main}
For all $N\geq 3$ and $d\geq 2N-1$, 
  the cohomology group $\coh k{X_{d,n}}$ vanishes  for $(n+1)^2<k< N$.
\end{lem}

\begin{rem}\label{betterbounds}
Vassiliev proved that the behaviour of configurations with more than $k+1$ points contained in a $k$-dimensional linear subspace of $\Pp n$ does not play a role in his method (see e.g. \cite[Lemma~4 and~9]{Vart}). For this reason, one can replace the bound $2N-1$ with the minimal value $d_{0,N}$ of $d$ such that the restriction of \eqref{vbundle} to the locus where the  $p_i$ are in general linear position is a vector bundle of the desired rank. This allows to give a better bound on the degree $d$ in Lemma~\ref{main}.
\end{rem}
In the next two sections, we will prove Lemma~\ref{main}. Instead than studying the cohomology of $X_{d,n}$ directly, in the spirit of Vassiliev's method we will focus on its Alexander dual, the Borel--Moore homology of the discriminant $\Sigma_{d,n}$. To this end, we construct a cubical space $\calX$ based on a (simplified) resolution of the singularities of $\Sigma_{d,n}$, and show that its geometrical realization $\ba\calX\ba$, endowed with an appropriate topology, is proper homotopy equivalent to $\Sigma_{d,n}$. This is done in section~\ref{cubical}. In section~\ref{proof}, we define a stratification of $\ba\calX\ba$ by locally closed subsets and study the associated spectral sequence in Borel--Moore homology, usually called the \emph{Vassiliev spectral sequence}. We will use an explicit description of the strata to prove the vanishing of the $E^1$ terms of Vassiliev spectral sequence in a suitable range to prove Lemma~\ref{main}.

\section{Cubical resolutions}\label{cubical}

As usual with Vassiliev's method, the first step is to observe that knowing the cohomology of $X_{d,n}$ is equivalent to knowing the Borel--Moore homology of the discriminant $\Sigma=\Sigma_{d,n}$. This follows from Alexander duality:
\begin{equation}\label{alex}
\coh k {X_{d,n}}\cong\BM{2\ccc dn-k-1}{\Sigma_{d,n}}\otimes\Q\left(-\ccc dn \right),\ k>0,
\end{equation}
for $\ccc dn = \dim_\C V_{d,n}=\cc dn$.

To compute the Borel--Moore homology of $\Sigma$, we construct a cubical resolution of it. 
\begin{ntn}
For all $k\geq 0$, we denote by $\Sigma_{\geq k}$ the locus inside $\Sigma$ of polynomials whose singular locus contains at least $k$ distinct points.

We denote the symmetric group in $k$ letters by $\s_k$,
 the space of ordered configurations of $k$ points in $\Pp n$ by $\F k{\Pp n}=\{(p_1,\dots,p_k)\in(\Pp n)^k|\; p_i\neq p_j\text{ for }i\neq j\}$
and the corresponding space of unordered configurations by
$\B k{\Pp n}=\F k{\Pp n}/\s_k$.
\end{ntn}

It is easy to show that $\Sigma$ is singular in codimension $1$. In particular, polynomials $f$ defining hypersurfaces with cusps or with more than one singular point belong to the singular locus of $\Sigma$. 
However, there is a natural way to construct a resolution of singularities of $\Sigma$, i.e. a proper surjective map $\phi:\;\tilde\Sigma\rightarrow \Sigma$ from a non-singular quasi-projective variety $\tilde\Sigma$ to $\Sigma$ which restricts to an isomorphism $\tilde\Sigma\setminus\phi^{-1}(\Sing(\Sigma))\cong (\Sigma\setminus\Sing(\Sigma))$ on the complement of the singular locus. Namely, one can consider the following map:
$$
\cusub\calX1:=\tilde\Sigma:=\{(f,p)\in V_{d,n}\times\Pp n|\; p \in \Sing(f)\} \xrightarrow{\phi} \calX_\emptyset:=\Sigma.
$$

This resolution is a homeomorphism outside the set $\Sigma_{\geq2}$ and its preimage $\phi^{-1}(\Sigma_{\geq 2})$. 
Furthermore, to construct a resolution of singularities of $\phi^{-1}(\Sigma_{\geq 2})$ is not difficult: it suffices to consider the space of triples $(f,p_1,p_2)$ where the $p_i$ are prescribed singular points of $f$, with its natural forgetful map to $\cusub\calX1$. Of course, ordered pairs $(p_1,p_2)$ are equivalent to considering inclusions $x_1\subset x_2$ with $x_1=\{p_1\}\in\B1{\Pp n}$ and $x_2=\{p_1,p_2\}\in\B2{\Pp n}$.
Again, this resolution of singularities is a homeomorphism outside the preimage $\phi^{-1}(\Sigma_{\geq3})$ of $\Sigma_{\geq3}$ in $\phi^{-1}(\Sigma_{\geq2})$.

By iterating  $N-1$ times this construction of a resolution,  one obtains the following spaces:
\begin{dfn}
For $1\leq i_1\leq \dots\leq i_r\leq N-1$ and $I=\{i_1,\dots,i_r\}$ we set
$$
\calX_I=\left\{(f,x_1,\dots,x_r)\in V_{d,n}\times \prod_{1\leq j\leq r} \B{i_j}{\Pp n}|\; x_1\subset x_2\subset\dots\subset x_r\subset \Sing f\right\}$$
and
$\calX_{I\cup\{N\}}=\left\{(f,x_1,\dots,x_r)\in \calX_I|\; f\in\overline{\Sigma_{\geq N}}\right\}$,
where $\overline{\Sigma_{\geq N}}$ denotes the Zariski closure of ${\Sigma_{\geq N}}$ in $V_{d,n}$.
In particular, we have $\calX_\emptyset = \Sigma$ and $\cusub\calX N = \overline{\Sigma_{\geq N}}$.
\end{dfn}

For each inclusion $I\subset J\subset\{1,\dots, N\}$ there is a natural forgetful map $\phi_{IJ}:\;\calX_J\rightarrow\calX_I$. 
This can be rephrased by saying that $\calX$ is  a \emph{cubical space} over the set $\{1,\dots,N\}$. For general background on cubical spaces and their relationship with semisimplicial spaces we refer to \cite[\S5.1.1]{ps-book}.

\begin{rem} 
The construction of the cubical space $\calX$ is directly inspired by the construction of cubical resolutions for pairs of algebraic varieties in \cite[\S{}I.2]{hyperres} and \cite[\S5.2]{ps-book}. 
We remark that in these references, at each step of the construction of the resolution, one considers the maximal subset on which the resolution of singularities is an \emph{isomorphism} of quasi-projective varieties and resolves the singularities of its complement in the next step. This has the advantage that the locus one is resolving is automatically Zariski closed.
Instead, we consider the locus where the resolution of singularities is a \emph{homeomorphism}, so its complement --- such as for example $\Sigma_{\geq2}$ in the case of $\tilde\Sigma\rightarrow\Sigma$ --- is {a priori} just a constructible subset and in general not Zariski closed. 
In particular, our maps $\phi_{\emptyset,I}:\;\calX_I\rightarrow\calX_\emptyset=\Sigma$ are (in general) not proper.
Note that the construction in \cite[Ch. 5]{ps-book} extends also to the case in which in the next step one takes a resolution of singularities of the Zariski closure of the complement of the locus in which the previously constructed resolution is not a homeomorphism, like $\overline{\Sigma_{\geq2}}$ in our example \cite[Rmk.~5.18]{ps-book}.  
\end{rem}

\begin{dfn}
For all $I\subset \{1,\dots,N\}$, we denote by $\Delta_I$ the simplex 
$$\Delta_I = \left\{(\alpha:\;I\rightarrow [0,1])\left| \; \sum_{i\in I}\alpha(i)=1\right.\right\}.$$
For every inclusion $I\subset J\subset\{1,\dots, N\}$ we denote by $e_{IJ}:\;\Delta_I\rightarrow\Delta_J$ the inclusion obtained by extending each function $\alpha\in\Delta_I$ to take value $0$ on $J\setminus I$.
\end{dfn}

\begin{dfn}
The \emph{geometric realization} of the cubical space $\calX$ is the quotient space 
$$
\ba\calX\ba =\left. \left(\bigsqcup_{I\subset\{1,\dots,N\}}\calX_I\times\Delta_I\right)\right/\sim
$$
where $\sim$ is the equivalence relation generated by $(\phi_{IJ}(t),\alpha)\sim(t,e_{IJ}(\alpha))$ for all $\alpha\in\Delta_I$, $t\in\calX_J$ with $I\subset J\subset \{1,\dots,N\}$.
\end{dfn}
Note that $\Delta_\emptyset$ is empty, so that $\calX_\emptyset$ does not play a role in the construction of the geometrical realization of $\calX$.

At this point, we need to define a topology on $\ba\calX\ba$ that takes degenerations appropriately into account. This is made necessary from our choice of working with subsets which were not Zariski closed in the definition of the cubical space $\calX$. 
We start by constructing a partial compactification of the spaces $\cusub\calX k$ for $1\leq k\leq N-1$. By definition, the elements of $\cusub\calX k$ are pairs $(f,x)$ with $x\in\B k{\Pp n}$ and $f\in V_{d,n}$ a polynomial singular in $x$. As we assumed $d\geq 2N-1$, being singular at the points in $x$ imposes $k(n+1)$ independent conditions on $f$, so that $\cusub\calX k$ is a vector bundle of rank $\cc dn-k(n+1)$ over $\B k{\Pp n}$. Moreover, the assumption $d\geq 2N-1$ implies that the fibres of $\cusub\calX k\rightarrow\B k{\Pp n}$ are pairwise distinct linear subspaces of $V_{d,n}$. In other words, the vector bundle structure on $\cusub\calX k$ induces an injection of $\B k{\Pp n}$ into the Grassmannian $G\left(\cc dn-k(n+1),V_{d,n}\right)$ of linear subspaces of $V_{d,n}$ of codimension $k(n+1)$. 
\begin{ntn}
For $1\leq k\leq N-1$, we denote by $L_k$ the Zariski closure of the image of the map $\B k{\Pp n}\hookrightarrow G(\dim V_{d,n}-k(n+1),V_{d,n})$ induced by the vector bundle $\cusub\calX k\rightarrow \B k{\Pp n}$. For $1\leq k_1\leq k_2\leq N-1$ and $\lambda_i\in L_{k_i}$ we write $\lambda_1<\lambda_2$ whenever we have inclusions $W_{\lambda_2}\subset W_{\lambda_1}$ of the corresponding linear subspaces in $V_{d,n}$.
\end{ntn}
Let us remark that this definition of $<$ agrees with the inclusion of configurations on the open subsets $\B k{\Pp n}\subset L_k$.

\begin{ntn}
For each $\lambda\in L_k$ the set-theoretical intersection of the singular loci of all $f$ lying in the corresponding linear subspace $W_\lambda\subset V_{d,n}$ is a non-empty subset which by the assumption $d\geq 2N-1\geq 2k-1$ contains at most $k$ distinct points in $\Pp n$. We will call this element the \emph{support} of $\lambda$ and denote it by $s(\lambda)\in \B{n(\lambda)}{\Pp n}$, where $1\leq n(\lambda)\leq k$ denotes the number of distinct points in $s(\lambda)$. 
\end{ntn}

We are ready to define a partial compactification $\balX$ of the cubical space~$\calX$. 

\begin{dfn}
For $1\leq i_1\leq \dots\leq i_r\leq N-1$ and $I=\{i_1,\dots,i_r\}$ we set
$$
\balX_I=\left\{(f,\lambda_1,\dots,\lambda_r)\in \Sigma\times \prod_{1\leq j\leq r} L_{i_j}|\; \lambda_1< \lambda_2<\dots< \lambda_r, f\in W_{\lambda_r}\subset V_{d,n} \right\}$$
and
$\balX_{I\cup\{N\}}=\left\{(f,\lambda_1,\dots,\lambda_r)\in \balX_I|\; f\in\overline{\Sigma_{\geq N}}\right\}$.
\end{dfn}
As in the case of $\calX$, the forgetful maps ${\bar\phi}_{IJ}:\;\balX_J\rightarrow\balX_I$ for $I\subset J$ define a structure of cubical space on $\balX$ with the property that all maps $\bar\phi_{\emptyset,I}:\;\balX_I\rightarrow\balX_\emptyset=\Sigma$ are proper. 
We can use the concept of support of elements in $L_k$ to define a contraction map $\ba\balX\ba\rightarrow\ba\calX\ba$, as follows.
\begin{dfn}
Let us denote by $\rho:\;\bigsqcup_{I\subset\{1,\dots,N\}}\balX_I\times\Delta_I\rightarrow\ba\calX\ba$ the map defined by mapping $\left(f,\left(\lambda_i,\alpha_i\right)_{i\in I}\right)\in\balX_I\times\Delta_I$ 
to the equivalence class of $\left(f,\left(x_j,\beta_j\right)_{j\in J}\right)\in\calX_J\times\Delta_J$ with $J=\{n(\lambda_i)|\;i\in I\}$ and $x_j=s(\lambda_j)$, $\beta_j=\sum_{i\in I| n(\lambda_i)=j}\alpha_i$ for all $j\in J$.
\end{dfn}

It is easy to check that the map $\rho$ is compatible with the equivalence relation $\sim$ on $\bigsqcup_{I\subset\{1,\dots,N\}}\balX_I\times\Delta_I$ and that the induced map $\ba\balX\ba\rightarrow\ba\calX\ba$ is the identity when restricted to $\ba\calX\ba\subset\ba\balX\ba$. 
\begin{ntn}
We denote by $\bax$ the geometrical realization $\ba\calX\ba$ endowed with the topology induced by the topology on $\bigsqcup_{I\subset\{1,\dots,N\}}\balX_I\times\Delta_I$ under $\rho$. 
\end{ntn}

\begin{lem}
The augmentation $\bax\rightarrow\Sigma$ defined by the natural forgetful map extends to a homotopy equivalence of their one-point compactifications. In particular, it induces an isomorphism on Borel--Moore homology.
\end{lem}
\proof
To prove the claim, it suffices to prove that the augmentation is a proper map with contractible fibres. %
For sufficiently nice spaces, this is enough to ensure that the augmentation defines a proper homotopy equivalence. For instance, this follows from combining Theorem 1.1 and 1.2 from \cite{lacher}, see also the discussion in \cite[\S2.1]{WJR} for more details on the required topological conditions.

The properness of the augmentation follows from the definition of the topology on $\bax$ and the fact that the natural maps $\balX_I\times\Delta_I\rightarrow\Sigma$ are proper maps for all $I\subset\{1,\dots,N\}$, since they are the composition of the projection map to $\balX_I$ (with compact fibre $\Delta_I$) with the proper map $\bar\phi_{\emptyset,I}$. 
Next, we describe explicitly the fibres of $\bax\rightarrow\Sigma$ and check that they are contractible. If $f\in\Sigma\setminus\overline{\Sigma_{\geq N}}$ has exactly $k$ singular points $p_1,\dots,p_k$, then one can show that the fibre over $f$ can be contracted to the equivalence class of the point $(f,\{p_1,\dots,p_k\}),k\mapsto 1)\in\cusub\calX k\times\cusub\Delta k$. Using barycentric subdivisions one actually shows that the fibre over $f$ is piecewise linearly isomorphic to the $(k-1)$-dimensional simplex $\cusub\Delta{p_1,\dots,p_k}$, where the vertices are given by the equivalence classes of the $k$ points $(f,p_i,1\mapsto 1)\in\cusub\calX1\times\cusub\Delta1$. 
If $f$ belongs to $\overline{\Sigma_{\geq N}}$, then the fibre over $f$ is a topological cone with vertex in the point $(f,N\mapsto1)\in\cusub\calX N\times\cusub\Delta N$.
This follows from the fact that the maps $\phi_{I,I\cup\{N\}}:\;\calX_{I\cup\{N\}}\rightarrow\calX_I$ in the cubical structure are induced by the inclusion $\overline{\Sigma_{\geq N}}\hookrightarrow\Sigma$. 
For instance, in the case of $0\in\overline{\Sigma_{\geq N}}$, in which the singular locus coincides with $\Pp n$, the fibre over $0$ can be described as the topological cone over the topological $(N-1)$st self-join of $\Pp n$. 
\qed
\section{Proof of Lemma~\ref{main}}\label{proof}

We study the Borel--Moore homology of $\bax$ using the following stratification into locally closed subsets $F_1,\dots,F_N$.
\begin{dfn}
For $l=1,\dots,N$, we denote by $F_l$ the locally closed subset of $\bax$ defined by 
$$
F_l
=\rho\left(\bigsqcup_{\max I =l}\balX_I\times\Delta_I\right).
$$
\end{dfn}

By the definition of $\rho$, each subset $F_l$ coincides with the image under $\rho$ of the union of the $\calX_I\times\Delta_I$ with $\max I=l$ inside $\bigsqcup_{\max I =l}\balX_I\times\Delta_I$.

The Vassiliev spectral sequence is the spectral sequence $E^r_{p,q}\Rightarrow\BM{p+q}\Sigma$ in Borel--Moore homology associated with the filtration $F_\pu$. Its $E^1$ term is given by $E^1_{p,q}=\BM{p+q}{F_p}$.

For $l<N$, the description of $F_l$ given in \cite[Prop. 2.7]{OTM4} applies (see \cite[Thm. 3, Lemma 1]{gorinov} for a proof).
Hence, the space $F_l$ is a non-orientable simplicial bundle over $\cusub\calX{l}$, which in turn is a vector bundle of rank $\cc{d}{n}-l(n+1)$ over $\B l{\Pp n}$. 
The fibre of $F_l\rightarrow\cusub\calX{l}$ is isomorphic to the interior $\cusub\Delta {p_1,\dots,p_l}^\circ$ of a simplex of dimension $l-1$.
As $\cusub\Delta {p_1,\dots,p_l}^\circ$ is a contractible space of real dimension $l-1$, its only non-trivial Borel--Moore homology group is $\BM{l-1}{\cusub\Delta {p_1,\dots,p_l}^\circ}\cong \Q$. Under any loop based at a point $\{p_1,\dots,p_l\}\in\B l{\Pp n}$, the orientation of $\BM{l-1}{\cusub\Delta {p_1,\dots,p_l}^\circ}$ changes according to the sign representation of the symmetric group $\s_l$ permuting $p_1,\dots,p_l$. %

Therefore, the Borel--Moore homology of the stratum $F_l$ is given by 
$$
\BM\pu{F_l} %
=\BM[\pm \Q]{\pu-2\ccc dn+2ln+l+1}{\B l{\Pp n}}\otimes\Q(-\ccc dn+l(n+1))
$$
for $\ccc dn=\cc dn$, where  $\pm\Q$  denotes the rank $1$ local system on $\B l{\Pp n}$ induced by the sign representation of $\s_l$, so that $\BM[\pm \Q]\pu{\B l{\Pp n}}$ is the $\s_l$-alternating part of the Borel--Moore homology of $\F l{\Pp n}$. We will refer to Borel--Moore homology with $\pm\Q$-coefficients as \emph{twisted} Borel--Moore homology. 

\begin{prop}[{\cite[Lemma~2]{Vart}}]\label{grass}
The twisted Borel--Moore homology of $\B l{\Pp n}$ is given by
$$\BM [\pm\Q]\pu{\B l{\Pp n}}= H_{\pu-l(l-1)}(G(l,\C^{n+1});\Q)\otimes\Q(l(l-1)/2),$$
where $G(l,\C^{n+1})$ denotes the Grassmannian of $l$-dimensional linear subspaces of $\C^{n+1}$.
\end{prop}

Let us remark that, although Vassiliev did not consider Hodge structures in \cite{Vart}, it follows from his proof of Proposition~\ref{grass} that the Hodge structures on $\BM [\pm\Q]k{\B l{\Pp n}}$ are pure of Tate type and weight $-k$. 

As $G(l,\C^{n+1})$ is empty for $l> n+1$, this means that $E^1_{p,q}$ vanishes for $n+1<p<N$, i.e. only the strata $F_1,\dots,F_{n+1}$ contribute to the Borel--Moore homology of $F_1\cup\dots\cup F_{N-1}=\bax\setminus F_N$. 

The description of the open stratum $F_N$ is more complicated, nevertheless the following lemma implies that its Borel--Moore homology does not contribute to the stable cohomology of $X_{d,n}$. The proof essentially consists in proving $E^1_{N,q}=0$ for the range $q\geq 2\left(\cc dn-N\right)$ in Vassiliev's spectral sequence.

\begin{lem}\label{laststratum}
\begin{align*}
\BM k\Sigma&\cong\BM k{\bax\setminus F_{N}} & \forall k\geq 2\cc dn -N+1.
\end{align*}
\end{lem}
\begin{proof}%

To prove the claim, we stratify $F_N$ as the union of the following locally closed substrata:
$$
\Phi_0:=\rho\left(\cusub\calX N\times\cusub\Delta N\right),
\ \ \ 
\Phi_l:=\rho\left(\bigsqcup_{\max J=l} \calX_{J\cup\{N\}}\times\Delta_{J\cup\{N\}}\right)\text{ for }1\leq l\leq N-1.
$$

By definition, $F_N$ and its substrata are only determined by the geometry of the inclusions $x_1\subsetneq x_2\subsetneq\dots\subsetneq x_r$ of subsets of the singular loci of elements $f\in\Sigma_{\geq N}$, where the $x_i$ consist of at most $N-1$ distinct points. For a substratum $\Phi_l$, the chain of inclusions should terminate with $x_r\in\B l{\Pp n}$. 
Keeping this in mind, it is easy to generalize \cite[Prop. 2.7]{OTM4} to show that the natural map 
$$
\Phi_l\longrightarrow \cusub\calX{l,N}
$$
is a locally trivial fibration whose fibre is the interior of an $l$-dimensional simplex. Intuitively, for $l\geq 1$ the fibre at $(f,x)\in\cusub\calX{l,N}\subset\cusub\calX l$ can be thought of as a cone over the fibre of $F_l\longrightarrow\cusub\calX l$. 

As we are dealing with polynomials $f$ with at least $N$ singular points, we have that $(f,p_1,\dots,p_N)\mapsto(f,\{p_1,\dots,p_l\})$ defines a surjection 
$$\left\{(f,p_1,\dots p_N)\in V_{d,n}\times\F N{\Pp n}|\;
p_1,\dots,p_N\in\Sing(f)\right\}
\longrightarrow \cusub\calX{l,N},
$$
where the domain is a vector bundle of rank $\cc dn-N(n+1)$ over $\F N{\Pp n}$ by the assumption $d\geq 2N-1$. 
As a consequence, the complex dimension of $\cusub\calX{l,N}$ is at most $\cc dn-N$ and the real dimension of each stratum $\Phi_l$ (in the sense of the maximal dimension of a cell in a cell decomposition of $\Phi_l$) is at most $2\cc dn - 2N+l$. In particular, the largest-dimensional stratum $\Phi_{N-1}$ has real dimension smaller than $2\cc dn -N-1$, which implies that the Borel--Moore homology of $F_N$  vanishes in degree larger than or equal to $2\cc dn-N$. 
Then the claim follows from the long exact sequence
$$
\dots\rightarrow\BM{\pu+1}{F_N}
\rightarrow\BM\pu{\bax\setminus F_N}
\rightarrow\BM\pu{\bax}
\rightarrow\BM\pu{F_N}\rightarrow\dots
$$
induced by the closed inclusion $\bax\setminus F_N\hookrightarrow \ba \calX\ba$.
\end{proof}

In view of the above lemma, we can concentrate on the first $n+1$ strata. 
As the Grassmannian $G(l,\C^{n+1})$ has complex dimension $l(n+1-l)$, 
in view of Proposition~\ref{grass}
the twisted Borel--Moore homology of $\B l{\Pp n}$ is non-trivial only between degree $l(l-1)$ and degree $2ln-l(l-1)$.
Hence, the Borel--Moore homology of $F_l$ can be non-trivial only between 
degree $l(l-1)+2\cc dn-2ln+1-l-1=2\cc dn-l(2n+2-l)-1$
and degree $2ln-l(l-1)+2\cc dn-2ln-l-1=2\cc dn-l^2-1$.

In particular, the minimal degree in which a stratum $F_l$ with $l\leq n+1$ has non-trivial Borel--Moore homology is degree 
$2\cc dn-(n+1)^2-1$, which is attained exactly for $l=n+1$.
Hence, this is the minimal degree for which the Borel--Moore homology of  $\bax\setminus F_N$ can be non-trivial. 
By Lemma \ref{laststratum}, this implies the vanishing of the $k$th Borel--Moore homology group of $\Sigma$ 
for $2\cc dn -N\leq k<2\cc dn-(n+1)^2-1$.
Then the claim of Lemma \ref{main} follows from Alexander duality \eqref{alex}.%

As a further check, one can consider Hodge structures and check that the Hodge weight of $\coh{(n+1)^2}{X_{d,n}}$ agrees with the Hodge weight $(n+1)(n+2)$ of $\eta_0\dots\eta_n$. 
One has $G(n+1,\C^{n+1})=\{\text{pt}\}$, so that the only non-trivial twisted Borel--Moore homology group of $\B {n+1}{\Pp n}$ is $$\BM[\pm\Q]{n(n+1)}{\B{n+1}{\Pp n}} =\Q(n(n+1)/2).$$ From this one obtains 
$$\BM{2\ccc dn-(n+1)(n+2)+n}{F_{n+1}}=\Q(\ccc dn-(n+1)(n+2)/2),\ \ccc dn=\cc dn$$
and  
$$\coh{(n+1)^2}{X_{d,n}}=\Q(-(n+1)(n+2)/2)
$$
after applying Alexander duality \eqref{alex}.

\def\cprime{$'$}

\end{document}